\documentclass[oneside]{amsart}
\usepackage{amssymb,amscd,amsthm,verbatim}
\usepackage{latexsym}
\usepackage{color}
\usepackage{xcolor}
\usepackage{graphicx}
\usepackage{mathrsfs}
\usepackage[all,cmtip]{xy}
\usepackage{bbm}
\usepackage{stmaryrd}
\usepackage{xr}
\usepackage{cite}

\usepackage[
	colorlinks,
	linkcolor={black},
	citecolor={black},
	urlcolor={black}
]{hyperref}

\date{\today}

\newcommand{\tr}{{\mathrm{Tr}\,}}

\newcommand{\calD}{{\mathcal{D}}}
\newcommand{\calS}{{\mathcal{S}}}

\newcommand{\bbA}{{\mathbb{A}}}
\newcommand{\bbC}{{\mathbb{C}}}
\newcommand{\bbH}{{\mathbb{H}}}
\newcommand{\bbM}{{\mathbb{M}}}
\newcommand{\bbN}{{\mathbb{N}}}
\newcommand{\bbP}{{\mathbb{P}}}
\newcommand{\bbR}{{\mathbb{R}}}
\newcommand{\bbT}{{\mathbb{T}}}
\newcommand{\bbZ}{{\mathbb{Z}}}

\newcommand{\frA}{{\mathfrak{A}}}
\newcommand{\frF}{{\mathfrak{F}}}
\newcommand{\frT}{{\mathfrak{T}}}

\newcommand{\rms}{{\mathrm{s}}}
\newcommand{\rmu}{{\mathrm{u}}}

\newcommand{\scrJ}{{\mathscr{J}}}

\newcommand{\DS}{{\calD\calS}}

\newcommand{\unstSect}{\mathsf{E}^\rmu}
\newcommand{\stabSect}{\mathsf{E}^\rms}

\newtheorem{theorem}{Theorem}[section]

\newtheorem{prop}[theorem]{Proposition}
\newtheorem{coro}[theorem]{Corollary}

\theoremstyle{definition}

\newtheorem{definition}[theorem]{Definition}

\newtheorem{remark}[theorem]{Remark}

\theoremstyle{plain}

\allowdisplaybreaks

\numberwithin{equation}{section}

\DeclareMathOperator{\supp}{supp}
\DeclareMathOperator{\sgn}{sgn}
\DeclareMathOperator{\Arg}{Arg}

\newcommand{\set}[1]{\left\{#1\right\}}

\definecolor{purple}{rgb}{.5,0,1}

\begin{document}

\author[D. Damanik]{David Damanik}
\address{Department of Mathematics, Rice University, Houston, TX~77005, USA}
\email{damanik@rice.edu}

\author[J. Fillman]{Jake Fillman}
\address{Department of Mathematics, Texas State University, San Marcos, TX 78666, USA}
\email{fillman@txstate.edu}

\author[Z. Zhang]{Zhenghe Zhang}
\address{Department of Mathematics, University of California, Riverside, CA 92521, USA}
\email{zhenghe@ucr.edu}

\title[Gap Labelling for Ergodic Jacobi Matrices]{Johnson--Schwartzman Gap Labelling for \\ Ergodic Jacobi Matrices}

\maketitle

\vspace{-1cm}

\begin{abstract}
We consider two-sided Jacobi matrices whose coefficients are obtained by continuous sampling along the orbits of a homeomorphim of a compact metric space. Given an ergodic probability measure, we study the topological structure of the associated almost sure spectrum. We establish a gap labelling theorem in the spirit of Johnson and Schwartzman. That is, we show that the constant value the integrated density of states takes in a gap of the spectrum must belong to the countable Schwartzman group of the base dynamics. This result is a natural companion to a recent result of Alkorn and Zhang, which established a Johnson-type theorem for the families of Jacobi matrices in question.
\end{abstract}

%\setcounter{tocdepth}{1}
%\tableofcontents

%changing the hyperlink-colors such that they are are black in the table of contents and thereafter dark blue
\hypersetup{
	linkcolor={black!30!blue},
	citecolor={red},
	urlcolor={black!30!blue}
}

\section{Introduction}

In this paper we are interested in the spectra of Jacobi matrices with dynamically defined coefficients, including the singular case where the off-diagonal entries are allowed to vanish. It is well known that the topological structure of these spectra can range from connected sets (i.e., sets without any interior gaps) to sets given by a finite union of non-degenerate compact intervals all the way to Cantor-type sets (i.e., sets with empty interior). A very useful tool that has been deployed in the Schr\"odinger case, which corresponds to the case where the off-diagonal elements of the Jacobi matrix in question are identically equal to one, is \emph{gap labelling theory}. Using the well-known fact that the set of growth points of the integrated density of states (IDS) coincides with the almost sure spectrum, and hence on each gap of the spectrum, the IDS takes a constant value that can be used to \emph{label} this gap in a unique fashion, gap labelling theory provides useful ways to interpret these values of the IDS in gaps, which are in fact heavily restricted by the base dynamics. Two common interpretations are based on $K$-theory \cite{Bel1986, Bel1990, Bel1992b, Bel2003, BelBovGhe1992} and the Schwartzman homomorphism \cite{Johnson1986JDE}, respectively. In either case, it follows that there is an at most countable set of values, which is completely determined by the base dynamics, such that for all continuous sampling functions used to generate the coefficients of the operator by sampling along orbits, the values of the IDS in gaps of the spectrum must belong to the countable set in question.

Both approaches to gap labelling theory have their advantages and disadvantages. The $K$-theory approach has a much broader scope, but explicit computations of the label set can be difficult. On the other hand, the approach based on the Schwartzman homomorphism has a more limited scope, but the computations of labels appear to be easier in this setting. We refer the reader to the recent survey \cite{DFGap} for background and more information. We also want to emphasize that the ability to carry out explicit computations in several cases of interest has led to various new results about spectra of dynamically defined Schr\"odinger operators; compare, for example, \cite{DFdm, DFGap}. In particular, examples have been found where the label set is as small as it can possibly be (equal to the integers), which has the immediate consequence that the spectrum has no interior gaps for any continuous sampling function.

The approach to gap labelling based on the Schwartzman homomorphism was developed in the Schr\"odinger case by Johnson \cite{Johnson1986JDE}. It uses as a critical input that there are continuous invariant sections for the associated energy-dependent cocycle for any energy in the complement of the spectrum. This input is one particular aspect of what is now called \emph{Johnson's theorem}, which characterizes the complement of the spectrum as the set of energies where an \emph{exponential dichotomy} holds. Johnson's theorem is proved in the Schr\"odinger case in the paper \cite{Johnson1986JDE} as well. However, a corresponding result in the Jacobi matrix case was not available until recently.

Alkorn and Zhang showed in \cite{AlkornZhang} that there is a version of Johnson's theorem for Jacobi matrices. Their work establishes such a result both in the case of  a fixed Jacobi matrix and in the case of a dynamically defined family of Jacobi matrices. In the latter setting, Marx had an earlier result under more restrictive assumptions \cite{Marx2014Nonlin}, whereas the Alkorn--Zhang result holds in complete generality. As \cite{AlkornZhang} and \cite{Marx2014Nonlin} have shown, the suitable replacement of the concept of exponential dichotomy in the case of Jacobi matrices is the concept of a \emph{dominated splitting}, which can be used to characterize the complement of the spectrum; see Subsection~\ref{ss.AZ} for more details. 

As Johnson's theorem and the Johnson--Schwartzman gap labelling are really companion results in the Schr\"odinger case, it is a natural goal to work out Johnson-Schwartzman gap labelling for Jacobi matrices, now that the Alkorn--Zhang extension of Johnson's theorem is available for Jacobi matrices. This is precisely what the present paper seeks to accomplish. 

We will state the desired result in Theorem~\ref{t.main} below. Let us first describe the setting in which it holds and the key quantities it involves.

A (whole-line) Jacobi matrix is an operator of the form
\begin{equation}
\label{eq:jacobidef}
[Ju](n) = \overline{a(n-1)} u(n-1) + b(n) u(n) + a(n) u(n+1),
\quad u\in \bbC^\bbZ, \ n \in \bbZ,\end{equation}
where $a(n) \in \bbC$ and $b(n) \in \bbR$ for $n \in \bbZ$. A half-line Jacobi matrix  on $\bbZ_+ = \{0,1,2,\ldots\}$ is given by \eqref{eq:jacobidef} for $n \in \bbZ_+$ with the boundary condition $u(-1)=0$. For $N \in \bbN$, an $N \times N$ Jacobi matrix acts on $\ell^2(\{0,1,\ldots,N-1\}) \cong \bbC^N$ via \eqref{eq:jacobidef} for $0 \leq n \leq N-1$ with boundary conditions $u(-1)=u(N)=0$.

 In the whole- and half-line cases, we will assume $a=\{a(n)\}$ and $b=\{b(n)\}$ are bounded sequences so that \eqref{eq:jacobidef} defines a bounded self-adjoint operator in $\ell^2(\bbZ)$ in the whole-line case and $\ell^2(\bbZ_+)$ in the half-line case. On one hand, this is a common assumption. On the other hand, it is forced by the particular setting in which we work, namely, that in which the sequences $a$ and $b$ are generated by continuously sampling points in a compact metric space.

Let us introduce this setting more precisely. Let $(\Omega,T)$ be a topological dynamical system, that is, $\Omega$ is a compact metric space and $T : \Omega \to \Omega$ is a homeomorphism. Given $q \in C(\Omega,\bbR)$, $p \in C(\Omega,\bbC)$, and $\omega \in \Omega$, we consider the family of Jacobi matrices $\{J_\omega\}_{\omega \in \Omega}$, acting in $\ell^2(\bbZ)$, given by 
\begin{equation} \label{eq:abomegadef}
a_\omega(n) = p(T^n\omega), \quad b_\omega(n) = q(T^n\omega), \quad n \in \bbZ,
\end{equation}
 that is,
\begin{equation}\label{e.jacmat}
[J_\omega\psi](n) = \overline{p(T^{n-1}\omega)}\psi(n-1)+q(T^n\omega)\psi(n) + p(T^n\omega)\psi(n+1).
\end{equation}
We will call a family $\{J_\omega\}$ as defined above a \emph{topological family of Jacobi matrices}.

While most of the paper works under these assumptions, we will also briefly discuss the case where $T$ is merely continuous, but not invertible. In this case one can define half-line Jacobi matrices, acting in $\ell^2(\bbZ_+)$, in the following way:
\begin{equation}\label{e.jacmathl}
[J_\omega\psi](n) = \begin{cases} \overline{p(T^{n-1}\omega)}\psi(n-1)+q(T^n\omega)\psi(n) + p(T^n\omega)\psi(n+1) & n \ge 1 \\ q(\omega)\psi(0) + p(\omega)\psi(1) & n = 0. \end{cases}
\end{equation}
We will explicitly say when we drop the assumption that $T$ is a homeomorphism.

Fix a $T$-ergodic Borel probability measure $\mu$ on $\Omega$. We assume 
\begin{equation} \label{eq:fullsupp}
\supp\mu = \Omega,
\end{equation} 
where $\supp \mu$ denotes the topological support of $\mu$, that is, the smallest closed set having full $\mu$-measure. Let us mention that assumption \eqref{eq:fullsupp} is non-restrictive, since one can always replace the dynamical system $(\Omega,T)$ by $(\supp\mu,T|_{\supp\mu})$. The \emph{density of states measure} (DOSM) is given by
\begin{equation}\label{e.DOSM}
\int f \, d\kappa 
= \lim_{N \to \infty} \int f\, d\kappa_{\omega,N}
:= \lim_{N\to\infty} \frac{1}{N} \tr(f(J_\omega\chi_{_{[0,N)}})), \quad \mu\text{-a.e.\ } \omega \in \Omega,
\end{equation}
and the \emph{integrated density of states} (IDS) is then given by
\begin{equation}\label{e.IDS}
k(E) = \int \! \chi_{_{(-\infty,E]}} \, d\kappa.
\end{equation}
The following statements are well known and not hard to check:
\begin{enumerate}
\item There is a fixed compact set $\Sigma = \Sigma_{\mu,p,q} \subseteq \bbR$ such that $\Sigma = \sigma(J_\omega)$ for $\mu$-a.e.\ $\omega \in \Omega$.
\item $\mu$-a.e.\ $\omega$ has a dense orbit in $\Omega$. One has $\Sigma = \sigma(J_\omega)$ for any $\omega$ with a dense orbit.
\item $\supp(\kappa) = \Sigma$.
\item $k$ is constant on each connected component of $\bbR \setminus \Sigma$.
\end{enumerate}
In view of the items above, it is natural to associate to each gap of $\Sigma$ the constant value assumed by $k$ on the gap; this value is called the \emph{label} of the gap. Gap-labelling theory attempts to characterize valid labels of gaps in terms of the topology and dynamics of the underlying system.

Our main result is the Johnson--Schwartzman gap-labelling theorem for ergodic Jacobi matrices.

\begin{theorem}[Gap Labelling for Ergodic Jacobi Matrices]\label{t.main}
Let $\{J_\omega\}_{\omega \in \Omega}$ be a topological family of Jacobi matrices defined over the ergodic topological dynamical system $(\Omega,T,\mu)$ with $\supp \mu = \Omega$. For all $E \in \bbR \setminus \Sigma$, $k(E)$ belongs to $\frA(\Omega,T,\mu) \cap [0,1]$, where $\frA(\Omega,T,\mu)$ denotes the Schwartzman group of $(\Omega,T,\mu)$.
\end{theorem}

\begin{remark}
In the event that $p(\omega)>0$ for all $\omega \in \Omega$, one can prove Theorem~\ref{t.main} in a very similar fashion to Johnson's arguments in \cite{Johnson1986JDE}. Let us emphasize that Theorem~\ref{t.main} contains no assumption on $p$ at all. In particular, $p$ need not be uniformly positive, log-integrable, or even nonzero a.e.\ with respest to the ergodic measure.
\end{remark}

For the reader's convenience, let us briefly recall the construction of the Schwartzman group; we direct the reader to the original papers \cite{Johnson1986JDE, Schwarzmann1957Annals} and the survey \cite{DFGap} for additional details. Let $(\Omega,T,\mu)$ be an ergodic topological dynamical system. The suspension of $(\Omega,T,\mu)$, denoted $(X,\tau,\nu)$, is given by
\begin{align*}
X& = \Omega \times [0,1] / ((\omega,1) \sim (T\omega,0)) \\
\int_X f\, d\nu & = \int_0^1 \int_\Omega f([\omega,t]) \, d\mu(\omega) \, dt,
\end{align*}
and $\tau$ denotes the translation flow in the second factor of $X$. Let $C^\sharp(X,\bbT)$ denote the set of equivalence classes in $C(X,\bbT)$ modulo homotopy. The Schwartzman homomorphism $\frF_\nu:C^\sharp(X,\bbT) \to \bbR$ is defined by
\begin{equation}
\frF_\nu([\phi]) = \lim_{t\to \infty} \frac{\widetilde\phi_x(t)}{t}, \quad \nu\text{ a.e.\ } x,
\end{equation}
where $\widetilde\phi_x$ is any lift of $\phi_x:t \mapsto \phi(\tau^tx)$ to a map $\bbR \to \bbR$. The Schwartzman group is then the range of this homomorphism:
\begin{equation}
\frA(\Omega,T,\mu) = \frF_\nu(C^\sharp(X,\bbT)).
\end{equation}

For later use, we also note that it is sometimes helpful to identify the circle $\bbT$ with the real projective line $\bbR\bbP^1$ and to use maps into $\bbR\bbP^1$ to characterize the range of the Schwartzman homomorphism. Concretely, identifying $\theta \in \bbT$ with $\mathrm{span}\{(\cos\pi\theta,\sin\pi\theta)^\top\}$ in $\bbR\bbP^1$, if $\phi \in C(X,\bbR\bbP^1)$, then one has
\begin{equation} \label{eq:schwartzmanArgument}
\frF_\nu([\phi]) = \lim_{t\to\infty} \frac{1}{\pi t} \Delta_{\Arg}^{[0,t]}\phi_x, \quad \nu\text{ a.e.\ } x,
\end{equation}
where $\Delta_{\Arg}^{[a,b]}\phi_x$ denotes the net change in the argument of $\phi_x$ on the interval $[a,b]$ and $\phi_x(t) = \phi(\tau^tx)$ as before. Naturally, one has
\[\Delta_{\Arg}^{[0,t]}\phi_x = \widetilde\phi_x(t) - \widetilde\phi_x(0)\]
where $\widetilde\phi_x$ is any lift of $\phi_x$.

For ease of reference, let us make some applications of Theorem~\ref{t.main} to specific choices of base dynamics explicit. The three examples we discuss are of wide interest and the spectra of the associated Schr\"odinger operators and some Jacobi matrices have been studied extensively; compare, for example, \cite{Damanik2017ESOSurvey, MarxJito2017ETDS} and references therein. The computation of the relevant Schwartzman group for each of them is covered by the following general result from \cite{DFGap}. Consider $T_{A,b} : \bbT^d \to \bbT^d, \; \omega \mapsto A \omega + b$, where $A \in \mathrm{SL}(d,\bbZ)$, $b \in \bbT^d$. Suppose $\mu$ is $T_{A,b}$-ergodic with $\supp(\mu) = \bbT^d$. Then, according to \cite[Theorem~8.1]{DFGap} we have that
\begin{equation}\label{e.DFatm}
\frA(\bbT^d,T_{A,b},\mu) = \{ mb + n : n \in \bbZ \text{ and } m \in  \bbZ^d \cap \ker(I-A^*) \}.
\end{equation}
Based on this result in combination with Theorem~\ref{t.main} we obtain the following corollaries. 

\begin{coro}[Gap Labelling for Quasi-Periodic Jacobi Matrices]
Let $\{J_\omega\}_{\omega \in \Omega}$ be a topological family of Jacobi matrices defined over a minimal translation on a finite-dimensional torus: $\Omega = \bbT^d$, $d \in \bbN$, $T = T_\alpha : \bbT^d \to \bbT^d$, $\omega \mapsto \omega + \alpha$, where $\alpha \in \bbT^d$ has rationally independent entries, $\mu = \mathrm{Leb}$. Then, for all $E \in \bbR \setminus \Sigma$, $k(E)$ belongs to $\bbZ^d \alpha + \bbZ$.
\end{coro}

\begin{proof}
Setting $A = I$, the identity matrix, and $b = \alpha$, we observe that $T_{A,b} = T_\alpha$, and hence the statement follows from Theorem~\ref{t.main} and \eqref{e.DFatm} since $\bbZ^d \cap \ker(I-A^*) = \bbZ^d$.
\end{proof}

\begin{coro}[Gap Labelling for Jacobi Matrices Generated by the Skew Shift]
Let $\{J_\omega\}_{\omega \in \Omega}$ be a topological family of Jacobi matrices defined by the standard skew shift: $\Omega = \bbT^2$, $T = T_{\mathrm{ss}} : \bbT^2 \to \bbT^2, \; (\omega_1, \omega_2)^\top \mapsto (\omega_1 + \alpha, \omega_1 + \omega_2)^\top$, where $\alpha \in \bbT$ is irrational, $\mu = \mathrm{Leb}$. Then, for all $E \in \bbR \setminus \Sigma$, $k(E)$ belongs to $\bbZ \alpha + \bbZ$.
\end{coro}

\begin{proof}
Setting
$$
A = \begin{bmatrix} 1 & 0 \\ 1 & 1 \end{bmatrix} \quad \text{and} \quad b = \begin{bmatrix} \alpha \\ 0 \end{bmatrix},
$$
we observe that $T_{A,b} = T_{\mathrm{ss}}$, and hence the statement follows from Theorem~\ref{t.main} and \eqref{e.DFatm} since $\bbZ^2 \cap \ker(I-A^*) = \bbZ \oplus \{ 0 \}$.
\end{proof}

\begin{coro}[Gap Labelling for Jacobi Matrices Generated by the Cat Map]
Let $\{J_\omega\}_{\omega \in \Omega}$ be a topological family of Jacobi matrices defined by the cat map: $\Omega = \bbT^2$, $T = T_{\mathrm{cm}} : \bbT^2 \to \bbT^2, \; (\omega_1, \omega_2)^\top  \mapsto (2\omega_1 + \omega_2, \omega_1 + \omega_2)^\top$, $\mu = \mathrm{Leb}$. Then, for all $E \in \bbR \setminus \Sigma$, $k(E)$ belongs to $\bbZ$. In particular, for all choices of continuous sampling functions $q\in C(\bbT^2,\bbR)$ and $p \in C(\bbT^2,\bbC)$, the associated almost sure spectrum $\Sigma$ is connected, that is, it has no interior gaps.
\end{coro}

\begin{proof}
Setting
$$
A = \begin{bmatrix} 2 & 1 \\ 1 & 1 \end{bmatrix} \quad \text{and} \quad b = \begin{bmatrix} 0 \\ 0 \end{bmatrix},
$$
we observe that $T_{A,b} = T_{\mathrm{cm}}$, and hence the first statement follows from Theorem~\ref{t.main} and \eqref{e.DFatm} since $\ker(I-A^*)$ is trivial.

The second statement following immediately from the first because on any interior gap of $\Sigma$, the IDS must take a value strictly between $0$ and $1$, which is impossible because only integers are allowed values.
\end{proof}

Schwartzman groups are computed for several additional examples in \cite{DFGap} and give, in combination with Theorem~\ref{t.main}, gap labelling results for the topological families of Jacobi matrices defined over them.

Additionally, the recent work \cite{DFdm} discussed Schr\"odinger operators generated by the doubling map on the circle and showed that the associated almost sure essential spectra are always connected. We show here that this result extends to the  Jacobi matrix case when the off-diagonals are bounded away from zero. The doubling map $T_\mathrm{dm} : \bbT \to \bbT$ is given by $T \omega = 2 \omega$. The normalized Lebesgue measure $\mu = \mathrm{Leb}$ on $\bbT$ is known to be ergodic. As $T_\mathrm{dm}$ is continuous but not invertible, we can proceed as above and associate for given $q \in C(\bbT,\bbR)$, $p \in C(\bbT,\bbC)$, and $\omega \in \bbT$ half-line Jacobi matrices as in \eqref{e.jacmat} by setting $a_\omega(n) = p(T_\mathrm{dm}^n \omega)$ and $b_\omega(n) = q(T_\mathrm{dm}^n \omega)$ for $n \in \bbZ_+$. There is now an almost sure essential spectrum, $\Sigma$, that is, for $\mu$-almost every $\omega \in \bbT$, we have $\sigma_\mathrm{ess}(J_\omega) = \Sigma$. Under the additional assumption that the range of the off-diagonal sampling function $p$ is contained in $\bbC^*=\bbC\setminus \{0\}$, this set does not have any gaps, as the following theorem shows.

\begin{theorem}[Absence of Spectral Gaps for Doubling Map Jacobi Matrices]\label{t.doublingmap}
Consider the doubling map $T_\mathrm{dm} : \bbT \to \bbT$. Then, for every $p \in C(\bbT,\bbC^*)$ and every $q \in C(\bbT,\bbR)$, the associated Lebesgue almost sure essential spectrum $\Sigma$ of the associated half-line Jacobi matrices is connected.
\end{theorem}

\begin{remark}
(a) For the Schr\"odinger case, $q \equiv 1$, this result was obtained in \cite{DFdm}.

(b) A similar statement holds for general linear expanding maps of the circle, $\omega \mapsto m \omega$, with an integer $m \ge 2$. The proof is analogous to the proof we give in the special case $m = 2$.

(c) Our proof needs the assumption that the off-diagonal sampling function avoids zero, and hence the result is limited to the case of Jacobi matrices whose off-diagonals are bounded away from zero. On a technical level, the proof of the result crucially uses the stable section of the associated cocycle at an energy in a gap of the essential spectrum, as it is this section that is independent of the past and hence can be defined for non-invertible base dynamics. On the other hand, Johnson-Schwartzman gap labelling requires a framework associated with an invertible dynamical system and hence in order to invoke it one needs to pass from the doubling map to the Smale solenoid. One then does have an unstable section, but as this section is no longer independent of the past, the more intricate topology of the solenoid places fewer restrictions on the rotation number. In particular, using the unstable section one is unable to see why the rotation number must be an integer. If $p$ takes the value zero, the stable section cannot be used to evaluate the Schwartzman homomorphism, and hence this case eludes our argument. We will revisit this discussion and give more details in Remark~\ref{r.concludingrem}, after having given the proof of Theorem~\ref{t.doublingmap}.
\end{remark}

\subsection*{Acknowledgements} D.D.\ was supported in part by NSF grants DMS--1700131 and DMS--2054752. J.F.\ was supported in part by Simons Foundation Collaboration Grant \#711663. Z.Z.\ was supported in part by NSF grant DMS--1764154. The authors thank the American Institute of Mathematics for hospitality and support through the SQuaRE program during a remote meeting in January~2021 and a January~2022 visit, during which part of this work was performed. The authors also gratefully acknowledge support from the Simons Center for Geometry and Physics, at which some of this work was done.

\section{Preliminaries}

\subsection{Oscillation Theory for Jacobi Matrices}

In this subsection we briefly recall oscillation theory for finite and half-line Jacobi matrices with positive off-diagonal entries. This does not use the underlying dynamics, so we leave the ergodic setting and consider a deterministic Jacobi matrix. For $N \in \bbN$, we denote by $\scrJ_N$ the set of $N \times N$ Jacobi matrices with $a(n) >0$ for all $0 \le n < N-1$. We write $\scrJ_\infty$ for the half-line Jacobi matrices with $a(n)>0$ for all $n \in \bbZ_+$.

Let us first consider the case $N = \infty$. Given $E \in \bbR$, consider the solution $u_E$ of $Ju=E u$ satisfying $u(0)=1$ and  interpolate linearly between consecutive integers to get a function $u_E:[0,\infty) \to \bbR$. For $m \in \bbN$, let 
\begin{equation}
F_m(E) = \#\{x \in (0,m) : u_E(x) = 0\},
\end{equation}
that is, $F_m(E)$ denotes the number of zeros of $u_E$  in the open interval $(0,m)$. On the other hand, if $N<\infty$, choose $a(N-1)>0$ arbitrarily and use this to define $u_E(n)$ for $0 \le n \le N$ so that $u_E(-1)=0$, $u_E(0)=1$ and
\begin{equation}
a(n-1)u_E(n-1) + b(n)u_E(n) + a(n) u_E(n+1) = E u_E(n), \quad \forall \, 0 \le n \le N-1,
\end{equation}
which we note defines $u_E(n)$ for all $0 \le n \le N$. As before, interpolate linearly between consecutive integers and let $F_m(E)$ denote the number of zeros of the interpolated $u_E$ in the interval $(0,m)$.

For any $1 \le m < N+1$, define $J_m$ to be the restriction of $J$ to $[0,m)\cap \bbZ$, that is,
\begin{equation}\label{e.finiteJMblock}
J_m =
\begin{bmatrix}
b(0) & a(0) \\
a(0) & b(1) & \cdots \\
&\cdots & \cdots & \cdots \\
&& \cdots & b(m-2) & a(m-2) \\
&&& a(m-2) & b(m-1)
\end{bmatrix}.
\end{equation}

The \emph{oscillation theorem} establishes a relationship between these objects: the number of eigenvalues of $J_m$ that exceed $E$ is precisely $F_m(E)$.

\begin{theorem}[Oscillation Theorem] \label{t:osc}
Let $N \in \bbN \cup \{\infty\}$ and $J \in \scrJ_N$ be given.
For any $E$ and any $m \in \bbN$ with $1 \le m < N+1$,
\begin{equation}\label{e.oscillation}
F_m(E) = \#[\sigma(J_m) \cap (E,\infty)].
\end{equation}
\end{theorem}

\begin{proof}
This is a well-known result; see e.g., \cite[Section~2]{Simon2005:OscTh} or \cite[Chapter~4]{Teschl2000:Jacobi}.
\end{proof}

\begin{remark}\label{r.oscillation}
Let us point out the following: the right-hand side of \eqref{e.oscillation} is fully determined by $J_m$ and $E$, and it is in particular independent of $a(m-1)$. On the other hand, the left-hand side of \eqref{e.oscillation} seemingly depends on $a(m-1)$, as the solution $u_E$ used in the definition of $F_m(E)$ depends on it. The resolution of this conundrum is that, while the solution on the interval in question depends on $a(m-1)$, the number of zeros of its interpolation does not, so long as $a(m-1) > 0$. This simple observation will play a key role in our proof of Theorem~\ref{t.main} in cases where the off-diagonal terms vanish and the infinite Jacobi matrix splits into finite blocks. These blocks will then naturally take the form \eqref{e.finiteJMblock} and hence define an expression as in the right-hand side of \eqref{e.oscillation}. We can then equate this expression via an application  of Theorem~\ref{t:osc} with the left-hand side of \eqref{e.oscillation} by simply pretending that the block in question is continued to the right with positive off-diagonal terms!
\end{remark}

\subsection{The Alkorn--Zhang Extension of Johnson's Theorem}\label{ss.AZ}

In this subsection we briefly summarize some relevant material from the paper \cite{AlkornZhang} by Alkorn and Zhang, which proved a very general version of Johnson's theorem for Jacobi matrices, both deterministic and dynamically defined. We begin with the definition of $\bbM(2,\bbC)$-cocycles that enjoy a dominated splitting and then state a result that uses this notion to characterize the set of energies outside the spectrum of a dynamically defined Jacobi matrix for an initial point that has a dense orbit. 

Let $\Omega$ be a compact metric space $\Omega$,  $T$ be a homeomorphism $\Omega \to \Omega$, and $B\in C(\Omega,\bbM(2,\bbC))$ be a continuous cocycle map, where $\bbM(2,\bbC)$ denotes the set of $2 \times 2$ matrices with complex entries. Iterates of the cocycle are given by
$$%\begin{equation}
B_0(\omega) = I, \quad B_n(\omega) = B(T^{n-1}\omega) \cdots B(T\omega)B(\omega), \quad n \in \bbN.
$$%\end{equation}  
In the following definition, we identify $z\in\bbC\bbP^1=\bbC\cup\{\infty\}$ with a one-dimensional subspace of $\bbC^2$ spanned by $(1,z)^\top$ and $\infty$ with the one spanned by $\vec{e}_2 := (0,1)^\top$.

\begin{definition}\label{d:domination_dynamical}
	Let $(\Omega,T)$ and $B$ be as above. Then we say $(T,B)$ has a \emph{dominated splitting} if there are two continuous maps $\stabSect, \unstSect:\Omega\to \bbC\bbP^1$ with the following properties:
	\begin{enumerate}
		\item $B(\omega)[\stabSect(\omega)]\subseteq \stabSect(T\omega)$ and $B(\omega)[\unstSect(\omega)]\subseteq \unstSect(T\omega)$ for all $\omega\in\Omega$.
		\item There are $N\in\bbZ_+$ and $\rho>1$ such that
		$$
		\|B_N(\omega)\vec u\|> \rho \|B_N(\omega)\vec s\|
		$$
		for all $\omega\in\Omega$ and all unit vectors $\vec u\in \unstSect(\omega)$ and $\vec s\in \stabSect(\omega)$.
	\end{enumerate}
\end{definition}

\begin{remark}\label{r:dyna_DS_implies}
	Condition~(2) above clearly implies that $B(\omega)\unstSect(\omega)\neq \{\vec 0\}$ for all $\omega\in\Omega$, which together with condition~(1) implies
	\begin{equation}\label{eq:nonzero_image_u}
	B(\omega)\unstSect(\omega)=\unstSect(T\omega)\mbox{ for all }\omega\in\Omega.
	\end{equation}
We also note that Condition~(2) clearly forces $\unstSect(\omega) \neq \stabSect(\omega)$ for all $\omega$ and hence
\[
\bbC^2 = \unstSect(\omega) \oplus \stabSect(\omega)
\]
for every $\omega \in \Omega$.
\end{remark}

It is often sufficient to identify some invariant continuous section, and Remark~\ref{r:dyna_DS_implies} shows that the unstable section $\unstSect$ can serve the desired purpose. This will indeed be sufficient for our proof of Theorem~\ref{t.main}. However, there are scenarios where it is crucial to also have the corresponding properties for the stable section $\stabSect$, and in particular
\begin{equation}\label{eq:nonzero_image_s}
B(\omega)\stabSect(\omega)=\stabSect(T\omega)\mbox{ for all }\omega\in\Omega.
\end{equation}
Indeed, our proof of Theorem~\ref{t.doublingmap} relies on this in an essential way.

A sufficient condition for the stable section to be truly invariant is given in the following proposition.

\begin{prop}\label{pr.stablesection}
Let $(\Omega,T)$ and $B$ be as above and suppose that $(T,B)$ has a dominated splitting in the sense of Definition~\ref{d:domination_dynamical}. If $\det B(\omega) \not= 0$ for every $\omega \in \Omega$, then the stable section is invariant, that is, it satisfies \eqref{eq:nonzero_image_s}.
\end{prop}

\begin{proof}
From the definition of dominated splitting, one has $B(\omega)\stabSect(\omega) \subseteq \stabSect(T\omega)$ for each $\omega$. The assumption $\det B(\omega) \neq 0$ implies that $B(\omega)\stabSect(\omega)$ is a one-dimensional subspace of the one-dimensional space $\stabSect(T\omega)$, which forces $B(\omega)\stabSect(\omega) = \stabSect(T\omega)$, as promised.
\end{proof}

In the setting of Jacobi matrices, we will consider the one-parameter family of cocycle maps (cf. \cite[Section~5]{AlkornZhang}) given by
\begin{equation} \label{eq:BlambdaCocycle}
B^E(\omega)
= \begin{bmatrix} E - q(\omega) & -\overline{p(T^{-1}\omega)} \\ p(\omega) & 0 \end{bmatrix}, \quad E \in \bbC, \ \omega \in \Omega.
\end{equation}
Let
\begin{equation}
\DS :=\{E \in\bbC: (T,B^E) \text{ has a dominated splitting}\}.
\end{equation}
Then we have \cite[Theorem~7]{AlkornZhang}:
\begin{theorem}\label{t.jacobi_johnson}
	Consider the Jacobi operators $J_\omega$, $\omega\in\Omega$, given by \eqref{e.jacmat}. Assume that $T$ is topologically transitive and let $\omega_0$ be any point that has a dense orbit. Then
	$$
	\rho(J_{\omega_0})= \DS.
	$$
	In particular, if $\mu$ is a fully supported ergodic measure on $\Omega$ and $\Sigma$ denotes the associated almost-sure spectrum, then  $\Sigma = \bbR \setminus \DS$.
\end{theorem}

Let us briefly discuss the relationship between $B^E$ and solutions of $J_\omega u = E u$, as this will be important later. We do not exhaustively discuss all possibilities, just the two that are relevant in the proof of the main results. First, if $p(T^n\omega) \neq 0$ for all $n \in \bbZ$, then one can readily check
\begin{equation} \label{eq:cocycleSols}
B^E(T^n\omega)\begin{bmatrix}u(n) \\ u(n-1) \end{bmatrix}
= p(T^n\omega) \begin{bmatrix}  u(n+1) \\ u(n)  \end{bmatrix}
\end{equation}
for each $n \in \bbZ$ and any two consecutive entries determine $u$ uniquely. On the other hand, if $p(T^{n_r}\omega)=0$ for $\cdots  n_{-1} < n_0 < n_1 < \cdots$ and $p(T^n\omega) \neq 0$ for all other $n \in \bbZ$, then $u$ is uniquely determined by $\{u(n_r+1)\}_{r \in \bbZ}$ and one still has \eqref{eq:cocycleSols} for all $n \notin \{n_r\}_{r \in \bbZ}$. For $n=n_r$, one has
\[B^E(T^{n_r}\omega) \begin{bmatrix} u(n_r) \\ u(n_r-1) \end{bmatrix} = \vec{0}.\]

\section{Proof of Main Theorem}

In this section we prove our main result, Theorem~\ref{t.main}. Throughout the discussion, we assume $(\Omega,T)$ is topologically transitive, $\omega_0$ has a dense orbit, and $\mu$ is a fully supported $T$-ergodic Borel probability measure on $\Omega$. Suppose $J_\omega = J_{p,q,\omega}$ is the ergodic Jacobi matrix generated by $\omega \in \Omega$ and the sampling functions $p,q$. 
\begin{remark} \label{rem:fullsupport}
As mentioned in the introduction, the assumption that $\mu$ is fully supported implies that $\mu$-a.e.\ $\omega \in \Omega$ has  dense $T$-orbit and hence the almost-sure spectrum $\Sigma$ coincides with the spectrum of any $J_\omega$ for which $\omega$ has a dense orbit. Moreover, by strong operator approximation, one has $\Sigma \supseteq \sigma(J_\omega)$ for all $\omega \in \Omega$.
\end{remark} 
The cocycle map $B^E$ is defined by \eqref{eq:BlambdaCocycle} and the spectrum of $J_{\omega_0}$ (and hence also the almost-sure spectrum of the family $\{J_\omega\}_{\omega \in \Omega}$) is characterized by Theorem~\ref{t.jacobi_johnson} as the set of $E \in \bbR$ for which $(T,B^E)$ does not enjoy a dominated splitting.

The first step in the proof of Theorem~\ref{t.main} is to reduce to the case of non-negative off-diagonal terms. To this end, we want to argue that we can pass from the off-line sampling function $p$ to the non-negative off-diagonal sampling function $|p|$. To express this fact, let us make the dependence of the IDS $k$ and the almost sure spectrum $\Sigma$ on the sampling functions $p$ and $q$ explicit in the following proposition.

\begin{prop}\label{p.takingabsval}
We have $k_{p,q} = k_{|p|,q}$ and $\Sigma_{p,q} = \Sigma_{|p|,q}$.
\end{prop}

\begin{proof}
Note first that the almost sure spectrum is determined by the IDS via $\Sigma_{p,q} = \supp(dk_{p,q})$, and hence $k_{p,q} = k_{|p|,q}$ implies $\Sigma_{p,q} = \Sigma_{|p|,q}$. We can therefore focus on the first identity.

Recall from \eqref{e.DOSM} that the DOSM $\kappa_{p,q}$ is the almost sure weak limit of the finitely supported measures that place point masses of weight $1/N$ at the eigenvalues of $J_{p,q,\omega}^{(N)} := \chi_{_{[0,N)}}J_{p,q,\omega}\chi_{_{[0,N)}}$ (counted with multiplicity) and that the IDS is the accumulation function associated with the DOSM, see \eqref{e.IDS}. Thus, the identity $k_{p,q} = k_{|p|,q}$ follows once we show that $J_{p,q,\omega}^{(N)}$ and $J_{|p|,q,\omega}^{(N)}$ are unitarily equivalent for every $N \in \bbN$ and every $\omega \in \Omega$.

To see that they are indeed unitarily equivalent, define $\{ \lambda_n\}_{n \in \bbZ_+}$, by
\begin{equation}\label{e.lambdandef}
\lambda_0 = 1 , \quad \lambda_{n+1} = \lambda_n e^{-i \Arg(p(T^n \omega))},
\end{equation}
with the convention $e^{-i \Arg(0)}=1$. Note that for each $n \in \bbZ_+$, we have $|\lambda_n| = 1$, and hence $\overline{\lambda_n} = \lambda_n^{-1}$, which will be used below.

A short calculation shows that with $\Lambda^{(N)} = \mathrm{diag}(\lambda_0, \ldots, \lambda_{N-1})$, we have
\begin{equation}\label{e.unitaryequiv}
(\Lambda^{(N)})^* J_{p,q,\omega}^{(N)} \Lambda^{(N)} = J_{|p|,q,\omega}^{(N)}.
\end{equation}
Indeed, for $0 \le m,n \le N-1$, we have
\begin{align*}
\langle \delta_m, (\Lambda^{(N)})^* J_{p,q,\omega}^{(N)} \Lambda^{(N)} \delta_n \rangle & = \begin{cases} q(T^n \omega) & m = n \\ \overline{\lambda_{n+1}} \lambda_{n} \overline{p(T^{n} \omega)} & m = n+1 \\ \overline{\lambda_{n-1}} \lambda_{n} p(T^{n-1} \omega) & m = n - 1 \\ 0 & |m-n| \ge 2 \end{cases} \\
& = \begin{cases} q(T^n \omega) & m = n \\ |p(T^{n} \omega)| & m = n+1 \\ |p(T^{n-1} \omega)| & m = n - 1 \\ 0 & |m-n| \ge 2 \end{cases} \\
& = \langle \delta_m, J_{|p|,q,\omega}^{(N)} \delta_n \rangle,
\end{align*}
where we used \eqref{e.lambdandef} in the second step. The assertion follows from \eqref{e.unitaryequiv} since $\Lambda^{(N)}$ is unitary.
\end{proof}

\begin{proof}[Proof of Theorem~\ref{t.main}]
Due to Proposition~\ref{p.takingabsval} we may assume without loss of generality that
\begin{equation}\label{e.pnonneg}
p \geq 0.
\end{equation}
By Theorem~\ref{t.jacobi_johnson},
\[
\sigma(J_{\omega})= \bbR \setminus \DS = \Sigma 
\]
for a.e.\ $\omega$ and moreover for any $\omega$ having a dense orbit.

Given $E \in \DS$, write $\unstSect,\stabSect: \Omega \to \bbC\bbP^1$ for the unstable and stable sections, respectively. Since $p$ is real-valued, $\unstSect,\stabSect$ are \emph{real} sections, that is, $\unstSect(\omega), \stabSect(\omega) \in \bbR\bbP^1$ for all $\omega$.

Recall from Remark~\ref{r:dyna_DS_implies} that the \emph{unstable section} $\unstSect$ is genuinely invariant in the sense that
\begin{equation}
B^E(\omega)\unstSect(\omega) = \unstSect(T\omega);
\end{equation}
compare also \cite[Lemma~10 and Remark~5]{AlkornZhang}. Thus (since we need to pass to an interpolated object on the suspension), it is crucial for us to work with $\unstSect$, in contrast to the Schr\"odinger setting, in which one can freely work with either section.

Now consider the suspension $X = \Omega \times [0,1] / ((\omega,1) \sim (T\omega,0))$. Define $B_t^E$  for $t \in [0,1]$ by
\begin{align}\label{e.interpolation}
B_t^E
=\begin{cases}  \begin{bmatrix} \cos\frac{\pi}{2}\theta(t) & - \sin\frac{\pi}{2}\theta(t) \\ \sin\frac{\pi}{2}\theta(t)  & \cos\frac{\pi}{2}\theta(t) \end{bmatrix} & 0 \le t \le 1/3 \\[7mm]
\begin{bmatrix} \mu(t)(E - q(\omega)) & -1 \\ 1 & 0 \end{bmatrix} & 1/3 \le t \le 2/3 \\[7mm]
\begin{bmatrix} E - q(\omega) & -[(1-\eta(t))+\eta(t)p(T^{-1}\omega)] \\ (1-\eta(t))+\eta(t)p(\omega) & 0 \end{bmatrix} & 2/3 \le t \le 1.
\end{cases}
\end{align}
where $\theta$, $\mu$, and $\eta$ increase continuously from $0$ to $1$ on their respective domains. Then define
$$
\Lambda(\omega,t) = B_t^E(\omega)\unstSect(\omega).
$$
Observe that $\Lambda(\omega,t)$ is continuous and well-defined with values in $\bbR\bbP^1$, and one has
\begin{equation}
\Lambda(\omega,1) = \Lambda(T\omega,0)
\end{equation}
by invariance. The important ingredients that lead to this observation are the following:
\begin{itemize}
	\item[(i)] If $0\le t<1$, $\Lambda(\omega,t)$ is well-defined since the determinant of $B_t^E$ cannot vanish, which follows from \eqref{e.pnonneg} and \eqref{e.interpolation}.
	\item[(ii)] If $t=1$, then $\Lambda(\omega,1) = B^E(\omega)\unstSect(\omega)=\unstSect(T\omega)=\Lambda(T\omega, 0)\in\bbR\bbP^1$; compare Remark~\ref{r:dyna_DS_implies}.
\end{itemize}

Thus, $\Lambda$ descends to give a well-defined map $\overline\Lambda:X \to \bbR\bbP^1$ to which we may apply the Schwartzman homomorphism. Let us note that one can equivalently define $X$ as the quotient of $\Omega \times \bbR$ by $(\omega,t+n) \sim (T^n\omega,t)$, so we may write $\Lambda(\omega,t)$ and $\overline\Lambda([\omega,t])$ for $t \in \bbR \setminus [0,1]$ below as this is sometimes convenient.

To conclude, it suffices to show that
\begin{equation} \label{eq:jacobiSchwartzgoal}
\frF_\nu\left(\left[\overline\Lambda \right]\right) = 1-k(E),
\end{equation}
since $\bbZ \subseteq \frA$ by standard reasoning (see, e.g., \cite{DFGap}). There are two cases to consider, based on the behavior of the off-diagonal generating function.
\medskip

\textbf{\boldmath Case 1: $p$ is positive $\mu$-a.e.} This part follows from the line of reasoning in the Schr\"odinger case by passing to a suitable full-measure set. To keep the paper more self-contained, let us sketch the main steps. The assumption $p>0$ $\mu$-a.e.\ implies that there is a set $\Omega_\star$ of full $\mu$-measure such that all off-diagonal elements of $J_\omega$ are strictly positive for each $\omega \in \Omega_\star$. For each $\omega \in \Omega_\star$, one can argue exactly as in the proof of \cite[Theorem~1.1]{DFGap} to see that
\begin{equation} \label{eq:regcase:argTo1-k}
\lim_{t \to \infty} \frac{1}{\pi t} \Delta_{\Arg}^{[0,t]}\Lambda(\omega,\cdot) = 1-k(E),
\end{equation}
where $\Delta_{\Arg}^{[a,b]}$ denotes the net change in argument on the interval $[a,b]$ and we pass to a full-measure subset one more time if necessary to ensure the existence of the limit (compare \eqref{eq:schwartzmanArgument} and surrounding discussion). For the reader's convenience, let us describe this in more detail. Given $\omega \in \Omega_\star$, consider $u_{E,\omega}$, the solution of $J_\omega u = Eu$ satisfying $u(-1)= 0$, $u(0)=1$, denote 
$$\vec{u}_{E,\omega}(n) = \mathrm{span}\{(u_{E,\omega}(n), u_{E,\omega}(n-1))^\top\} \in \bbR\bbP^1$$ for $n \in \bbZ$, and then put $\vec{u}_{E,\omega}(n+t) = B_t^E(T^n\omega) \vec{u}_{E,\omega}(n)$ for $0 \le t \le 1$. Note that this is well-defined on account of \eqref{eq:cocycleSols}, which implies
\[\vec{u}_{E,\omega}(n+1) = B^E(T^n\omega)\vec{u}_{E,\omega}(n) = B_1^E(T^n\omega)\vec{u}_{E,\omega}(n).\]
From the explicit form of the homotopy $B_t^E$, one can verify that
\begin{equation} \label{eq:regcase:flipsArg}
\left\lfloor\frac{1}{\pi} \Delta_{\Arg}^{[0,n+1]}u_{E,\omega}(\cdot) \right\rfloor
= \#\set{j \in \bbZ \cap [0,n) : \sgn u_{E,\omega}(j) \neq \sgn u_{E,\omega}(j+1) }
\end{equation}
for each $n \in \bbN$ such that $u_{E,\omega}(n-1) \neq 0$. To relate this back to $\Lambda$, note that our choice of $\omega \in \Omega_\star$ implies that $B^E(T^n\omega)$ is invertible for all $n \in \bbZ$, from which one deduces
\begin{equation} \label{eq:regcase:lambdatracksue}
\Delta_{\Arg}^{[0,t]}u_{E,\omega}(\cdot) = \Delta_{\Arg}^{[0,t]}\Lambda(\omega,\cdot) + O(1)
\end{equation}
as $t \to \infty$. Thus, \eqref{eq:regcase:argTo1-k} follows from \eqref{eq:regcase:flipsArg},  \eqref{eq:regcase:lambdatracksue}, \eqref{e.IDS}, and Theorem~\ref{t:osc}.

From \eqref{eq:regcase:argTo1-k}, one deduces immediately that
\begin{equation}
\lim_{t\to\infty} \frac{1}{\pi t} \Delta_{\Arg}^{[0,t]} \overline{\Lambda}(\tau^s x) = 1-k(E)
\end{equation}
for $\nu$-a.e.\ $x$.
In view of \eqref{eq:schwartzmanArgument}, this implies \eqref{eq:jacobiSchwartzgoal}.   
\medskip

\textbf{\boldmath Case 2: $p$ vanishes on a set of positive $\mu$-measure.} In this case, ergodicity implies that there exists $\Omega_\star^1 \subseteq \Omega$ of full $\mu$-measure such that $J_\omega$ is a direct sum of finite Jacobi matrices for all $\omega \in \Omega_\star^1$. There is also a full-measure set $\Omega_\star^2$ such that
\begin{equation} \label{eq:convergence}
\kappa_{\omega,N} \to \kappa \text{ weakly for all } \omega \in \Omega_\star^2. 
\end{equation}
Consider $\omega \in \Omega_\star:= \Omega_\star^1\cap \Omega_\star^2$ and choose $\cdots < n_{-1} < n_0 < n_1 < \cdots$ so that $a_\omega(n) = p(T^n \omega) = 0 $ if and only if $n= n_r$ for some $r$. Thus,
\begin{equation} \label{eq:splitting}
J_\omega = \bigoplus_r J_r,
\end{equation}
where $J_r = J_\omega|_{[n_r+1,n_{r+1}]}$ is a finite Jacobi matrix for each $r$. Indeed, writing $\ell_r = n_{r+1}-n_r$, $J_r$ is an $\ell_r \times \ell_r$ matrix. 

By assumption, $E \notin \Sigma$, which gives 
\begin{equation} \label{eq:EnotinsigmaJr}
E \notin \sigma(J_r)
\end{equation}
by Remark~\ref{rem:fullsupport}.

Without loss of generality, assume $n_0 = 0$. Notice that by \eqref{e.IDS}, \eqref{eq:convergence}, and \eqref{eq:splitting}, we have
\begin{equation} \label{eq:JohnsonSing1}
\lim_{N \to\infty} \frac{\sum_{r=0}^{N-1} \#[\sigma(J_r)\cap(E,\infty)]}{\sum_{r=0}^{N-1} \ell_r} 
= 1-k(E).
\end{equation}

For each block $J_r$, consider the Dirichlet solution $u_r$ defined on $[n_r,n_{r+1}+1]$ by $u_r(n_r) = 0$, $u_r(n_r+1)=1$, and 
\begin{equation}
\widetilde a(n-1)u_r(n-1) + b(n)u_r(n) + \widetilde  a(n)u_r(n+1) = E u_r(n), \quad n_r+1 \leq n \leq n_{r+1}
\end{equation}
where $\widetilde a(n) = a_\omega(n)$ for $n_r \le n < n_{r+1}$ and $\widetilde a(n_{r+1})=1$.\footnote{Note here that $u_r$ is a Dirichlet solution associated with a block for which $a(n_{r+1}) =1$, so $u_r$ is \emph{not} an eigenvector of the original block, $J_r$. Indeed, it cannot be such a solution because $E \notin \sigma(J_r)$ as noted in \eqref{eq:EnotinsigmaJr}.} Note that this is the point in the proof where Remark~\ref{r.oscillation} is relevant.

If 
\begin{equation}
f_r := \#\{n_r+1 \le j \le n_{r+1}: \sgn u_r(j) \neq \sgn u_r(j+1) \}
\end{equation} 
denotes the number of sign changes of $u_r$ on $[n_r+1,n_{r+1}+1] \cap \bbZ$, then Theorem~\ref{t:osc} implies
\begin{equation} \label{eq:JohnsonSing2}
f_r = \#[\sigma(J_r)\cap(E,\infty)].
\end{equation}

Notice that
\begin{equation} \label{eq:Euatsingularpoints}
\unstSect(T^{n_r + 1}\omega) = \mathrm{span}(\vec e_1)
\quad r \in \bbZ,
\end{equation}
which can be seen directly from the form of the cocycle, $a(n_r) = 0$, and $\unstSect\neq \{\vec 0 \}$. We claim that
\begin{equation}  \label{eq:singCaseMaingoal}
f_r = \frac{1}{\pi} \Delta_{\Arg}^{[n_r+1,n_{r+1}+1]} \Lambda(\omega,\cdot). \end{equation}
In view of \eqref{eq:JohnsonSing1}, \eqref{eq:JohnsonSing2}, and the definition of the Schwartzman homomorphism, \eqref{eq:singCaseMaingoal} implies \eqref{eq:jacobiSchwartzgoal}, so all that remains is to demonstrate \eqref{eq:singCaseMaingoal}.

First, notice that \eqref{eq:Euatsingularpoints} together with the explicit interpolated form of $B_t^E$ implies that the right-hand side of \eqref{eq:singCaseMaingoal} is always a nonnegative integer, that is, 
\begin{equation} \label{eq:singCase:argmultpi}
\left\lfloor \frac{1}{\pi} \Delta_{\Arg}^{[n_r+1,n_{r+1}+1]} \Lambda(\omega,\cdot) \right\rfloor =   \frac{1}{\pi} \Delta_{\Arg}^{[n_r+1,n_{r+1}+1]} \Lambda(\omega,\cdot).
\end{equation}
Next, observe from \eqref{eq:Euatsingularpoints} and induction that 
\begin{equation} \label{eq:case2:Euisdirichlet}
\unstSect(T^m\omega) = \mathrm{span}\{(u_r(m), u_r(m-1))^\top\}, \quad n_r+1 \leq m \leq n_{r+1}.
\end{equation} 
Finally, for $t \in [2/3,1]$ one has
\begin{align*}
B_t^E\left[B_{2/3}^E\right]^{-1} 
& = \begin{bmatrix} E - q(\omega) & -[(1-\eta(t))+\eta(t)p(T^{-1}\omega)] \\ (1-\eta(t))+\eta(t)p(\omega) & 0 \end{bmatrix}\begin{bmatrix} 0 & 1 \\ -1 & E - q(\omega)  \end{bmatrix} \\
& = \begin{bmatrix} (1-\eta(t))+\eta(t)p(T^{-1}\omega) & * \\ 0 & [1-\eta(t)]+\eta(t)p(\omega) \end{bmatrix},
\end{align*}
which (for $\eta(t) \in [0,1)$) preserves the half-planes $\bbH_\pm = \set{(x,y)^\top : \pm y > 0}$ and the semi-axes $\bbA_\pm = \set{(x,0)^\top : \pm x > 0}$. Consequently, by the same argument as in the %proof of \cite[Theorem~1.1]{DFGap},
previous case, one has
\begin{equation} \label{eq:singGoalstep}
\left\lfloor  \frac{1}{\pi} \Delta_{\Arg}^{[n_r+1,m+1]} \Lambda(\omega,\cdot) \right\rfloor = \#\{n_r+1 \le j \le m : \sgn u_r(j) \neq \sgn u_r(j+1) \}
\end{equation}
for any $n_r \leq m \leq n_{r+1}$ such that $u_r(m) \neq 0$. As discussed earlier, $E \notin \sigma(J_r)$, which implies $u(n_{r+1})\neq 0$. Thus, \eqref{eq:singCaseMaingoal} follows from \eqref{eq:singCase:argmultpi} and \eqref{eq:singGoalstep}. 
\end{proof}

\begin{remark}
(a) Let us remark that, although the proof in Case~1 follows the same lines as in the Schr\"odinger case since one passes to a full-measure set on which one can apply standard oscillation theory, the recent work \cite{AlkornZhang} is still a crucial new ingredient, since one needs the existence of the invariant sections for the case $\min p(\omega) = 0$.

(b) Let us comment also that we found the result in Case~2 rather serendipitous; prior to working out the proof, there were nontrivial reasons to be concerned. Let us describe this in more detail. On one hand, everything must be exact; one cannot give up any errors at any point, because, as soon as one gives up an $O(1)$ error on any block, the positive density of the set of zeros of the off-diagonal elements causes a nontrivial error in the thermodynamic limit defining the integrated density of states. On the other hand, in Case~1, one has to give up errors at two steps. First, one incurs an $O(1)$ error in approximating the (change in argument of the interpolation of) the Dirichlet solution with that of the  unstable section (cf. \eqref{eq:regcase:lambdatracksue}). Second, one incurs an error in approximating the number of sign flips of the Dirichlet solution by the number of half-rotations of the corresponding interpolated object (cf.\ \eqref{eq:regcase:flipsArg}).

Thus, there are two fortuitous circumstances that facilitate the proof in Case~2. First, the unstable section precisely (projectively) coincides with the Dirchlet solution on the blocks (eq.\ \eqref{eq:case2:Euisdirichlet}) and second, one has an \emph{exact} expression relating the change in argument of the interpolated object to the sign flips of the Dirichlet solution (eq.\ \eqref{eq:singCaseMaingoal}). Thus, all of the moving parts fit together surprisingly nicely.
\end{remark}

\section{Absence of Spectral Gaps for the Doubling Map}\label{sec.3}

In this section we give the proof of Theorem~\ref{t.doublingmap}. Recall that the theorem extends the main result from \cite{DFdm} from the case $p \equiv 1$ to $p \in C(\bbT, \bbC \setminus \{0\})$. We will follow the general strategy from \cite{DFdm} and explain the necessary changes. However, for the convenience of the reader, we briefly sketch all steps of the argument, even those that do not require any changes relative to \cite{DFdm}.

Suppose we are given sampling functions $p,q$ as in the statement of Theorem~\ref{t.doublingmap}. Sampling orbits of the doubling map $T$, this gives rise to the family of one-sided Jacobi matrices $\{ J_\omega \}_{\omega \in \bbT}$, and we are interested in the Lebesgue almost surely common essential spectrum $\Sigma_{p,q}$.  In order to show that it is connected, we will pass to an associated family $\{ \widetilde J_{\widetilde \omega} \}_{\widetilde \omega \in \widetilde \Omega}$ of two-sided Jacobi matrices, defined over an invertible extension $(\widetilde \Omega, \widetilde T)$ of the doubling map.

Let us begin by recalling the construction of the standard (Smale--Williams) solenoid; see \cite[Section~1.9]{BrinStuck2015Book} or \cite[Section~17.1]{KatokHassel1995Book} for additional background. Consider the solid torus $\frT$ given by
$$
\frT = \bbT \times \overline{\mathbb{D}}, \quad \text{where } \overline{\mathbb{D}} = \{ (x,y) \in \bbR^2 : x^2 + y^2 \le 1 \}.
$$
Choose $\lambda \in (0,1/2)$ and define the transformation $F : \mathfrak{T} \to \mathfrak{T}$ by
$$
F(\omega,x,y) = \left( 2 \omega, \lambda x + \frac12 \cos (2 \pi \omega), \lambda y + \frac12 \sin (2 \pi \omega) \right).
$$
Then the map $F$ is injective and the set
$$
S = \bigcap_{n = 0}^\infty F^n(\frT)
$$
is a compact $F$-invariant subset of $\frT$ on which $F$ is a homeomorphism; the set $S$ is called the (standard) \emph{solenoid}. Note that the projection to the first component, $\pi_1(\omega,x,y) = \omega$, sends $S$ onto $\bbT$ and the induced transformation on $\bbT$ is given by the doubling map, that is, $\pi_1 \circ F = T \circ \pi_1$. In other words, setting $\widetilde \Omega = S$, $\widetilde T = F|_S$ we obtain the desired invertible extension $(\widetilde \Omega, \widetilde T)$ of the doubling map. There is then a natural ergodic extension  of Lebesgue measure on $\bbT$ to the space $\widetilde \Omega$, which we denote by $\widetilde \mu$. In addition to being ergodic, $\widetilde \mu$ has full topological support, that is,
\begin{equation}\label{eq:tildemusupport}
\supp \widetilde \mu = \widetilde \Omega.
\end{equation}
The measure $\widetilde \mu$ can be viewed as both the  Sinai--Ruelle--Bowen measure and as the Bowen--Margulis measure. Locally, $\widetilde \mu$ is the product of the $(1/2,1/2)$-Bernoulli measure on the Cantor fibers and the Lebesgue measure on the circle. Consequently, $\widetilde \mu$ projects to Lebesgue measure under $\pi_1$:
\begin{equation}\label{e.bowenprojection}
(\pi_1)_*(\widetilde \mu) = \mu.
\end{equation}

Given the sampling functions $p$ and $q$ defined on $\bbT$, we consider sampling functions on the solenoid as follows:
\begin{equation}\label{e.tildefdep}
\widetilde p : \widetilde\Omega \to \bbR, \; (\omega,x,y) \mapsto p(\omega)
\end{equation}
and
\begin{equation}\label{e.tildefdeq}
\widetilde q : \widetilde\Omega \to \bbR, \; (\omega,x,y) \mapsto q(\omega).
\end{equation}
On the solenoid, the cocycle we may consider takes the form
$$
(\widetilde T, \widetilde B^{E}) : \widetilde \Omega \times \bbC^2 \to \widetilde \Omega \times \bbC^2, \quad (\widetilde \omega, \vec v) \mapsto (\widetilde T \widetilde \omega, \widetilde B^E (\widetilde \omega) \vec v),
$$
where 
\begin{equation}\label{e.solenoidcocycle}
\widetilde B^E : \widetilde \Omega \to \mathrm{GL}(2,\bbC), \quad \widetilde \omega \mapsto  \begin{bmatrix} E - \widetilde q(\widetilde T(\widetilde \omega)) & - \overline{\widetilde p(\widetilde \omega)} \\ \widetilde p(\widetilde T (\widetilde \omega)) & 0 \end{bmatrix}.
\end{equation}
Notice that relative to the standard definition, we have to shift the arguments here; see the comment after \eqref{e.dmcocycle} below for the reason.

We are now ready for the

\begin{proof}[Proof of Theorem~\ref{t.doublingmap}]
Let $p \in C(\bbT,\bbC^*)$ and $q \in C(\bbT,\bbR)$ be given, associate sampling functions on $\widetilde\Omega$ via \eqref{e.tildefdep} and \eqref{e.tildefdeq}, and consider the family 
\begin{equation}\label{e.tildeoper}\begin{split}
[\widetilde J_{(\omega,x,y)} \psi](n)
& = \overline{\widetilde p(  \widetilde T^{n-1}(\omega,x,y))}\psi(n-1) + \widetilde q(\widetilde T^n(\omega,x,y))\psi(n) \\ & \qquad\qquad\qquad  + \widetilde p(\widetilde T^n(\omega,x,y))\psi(n+1).
\end{split}
\end{equation}

We denote the associated density of states measure by $\widetilde \kappa_{\widetilde p,\widetilde q}$, the associated integrated density of states by $\widetilde k_{\widetilde p,\widetilde q}$, and the associated $\widetilde \mu$-almost sure spectrum of $\widetilde J_{\widetilde \omega}$ by $\widetilde \Sigma_{\widetilde p,\widetilde q}$. As before, we have
\begin{equation} \label{eq:sigmasuppdk}
\widetilde \Sigma_{\widetilde p,\widetilde q} = \supp \widetilde \kappa_{\widetilde p,\widetilde q}.
\end{equation}

It is not difficult to check that for $\widetilde \omega = (\omega,x,y) \in S$ and $n \in \bbZ_+$, we have 
\begin{equation} \label{e.potentialscoincide}
\widetilde p(\widetilde T^n \widetilde \omega)= p(T^n \omega) \quad \text{and} \quad \widetilde q(\widetilde T^n \widetilde \omega) = q(T^n \omega).
\end{equation}
This in turn implies that
\begin{equation} \label{eq:spectracoincide}
\widetilde \Sigma_{\widetilde p,\widetilde q} = \Sigma_{p,q}.
\end{equation}

Thus, the theorem will follow from \eqref{eq:sigmasuppdk} and \eqref{eq:spectracoincide}  once we show that for every 
$E \in \bbR \setminus \Sigma_{p,q} = \bbR \setminus \widetilde \Sigma_{\widetilde p,\widetilde q}$, we must have
\begin{equation}\label{e.goal:idsmustbeinteger}
\widetilde k_{\widetilde p,\widetilde q}(E) \in \bbZ,
\end{equation}
as this implies that either $E < \min \widetilde \Sigma_{\widetilde p,\widetilde q} = \min \Sigma_{p,q}$ or $E > \max \widetilde \Sigma_{\widetilde p,\widetilde q} = \max \Sigma_{p,q}$, so it follows that $\Sigma_{p,q}$ has no interior gaps.

Invoking Proposition~\ref{p.takingabsval} we can assume without loss of generality that $p$ is real-valued, as we have $\widetilde k_{\widetilde p,\widetilde q} = \widetilde k_{|\widetilde p|,\widetilde q}$ and $\widetilde \Sigma_{\widetilde p,\widetilde q} = \widetilde \Sigma_{|\widetilde p|,\widetilde q}$ (along with \eqref{e.tildefdep}). This will ensure that all stable and unstable sections below corresponding to the real energy $E$ will be real as well.

Since $E \in \bbR \setminus \Sigma_{p,q} = \bbR\setminus \widetilde \Sigma_{\widetilde p,\widetilde q}$, Theorem~\ref{t.jacobi_johnson} implies that  $(\widetilde T,\widetilde B^E)$ has a dominated splitting (recall $\widetilde B^E$ denotes the cocycles as in \eqref{e.solenoidcocycle}). As usual, denote the stable and unstable sections by $\stabSect,\unstSect:\widetilde\Omega \to \bbR\bbP^1$.

As we are now in the case of invertible matrices, Proposition~\ref{pr.stablesection} implies that the stable section is genuinely invariant, that is,
\begin{equation}
\widetilde B^E(\widetilde\omega)\stabSect(\widetilde\omega) = \stabSect(\widetilde{T}\widetilde{\omega}).
\end{equation}

Of course we have the related (non-invertible) cocycle $(T,B^E)$ defined by 
\begin{equation}\label{e.dmcocycle}
B^E : \Omega \to \mathrm{GL}(2,\bbR), \quad \omega \mapsto \begin{bmatrix} E - q(T\omega) & - p(\omega) \\ p(T\omega) & 0 \end{bmatrix}
\end{equation}
and $(T, B^E)^n = (T^n, B^E_n)$ for $n \in \bbZ_+$. Here it becomes clear why we had to shift in \eqref{e.solenoidcocycle} above, as $T^{-1}$ does not exist.

As a consequence of \eqref{e.potentialscoincide} we find that for $\widetilde \omega = (\omega,x,y) \in S$ and $n \in \bbZ_+$, we have
\begin{equation}\label{e.cocyclescoincide}
\widetilde B^E_n(\widetilde \omega) = B^E_n(\omega).
\end{equation}

Now let $X = X(\bbT,T)  = \bbT \times [0,1] / ((\omega,1) \sim (T \omega,0))$ be the suspension of the doubling map $(\bbT,T)$ and let $\widetilde X = X(\widetilde\Omega,\widetilde T)  = \widetilde \Omega \times [0,1] / ((\widetilde\omega,1) \sim (\widetilde T \widetilde\omega,0))$ be the suspension of the standard solenoid $(\widetilde\Omega,\widetilde T)$. Moreover, let $\widetilde \nu$ denote the suspension of $\widetilde\mu$. Notice that 
$$
\pi_X : \widetilde X \to X, \quad [(\omega,x,y),s] \mapsto [\omega,s]
$$ 
is continuous. Following the terminology of \cite{DFdm}, we say that $\widetilde\phi \in C(\widetilde{X},\bbT)$ \emph{factors through} $X$ if there is $\phi \in C(X,\bbT)$ such that $\widetilde\phi = \phi\circ \pi_X$, that is,
$\widetilde\phi([(\omega,x,y),s]) = \phi([\omega,s])$ for all $ (\omega,x,y) \in \widetilde\Omega$ and $s \in [0,1]$. The following was shown in \cite{DFdm}: 
\begin{equation} \label{e.solenoidThroughTCsharp}
\widetilde\phi \in C(\widetilde X,\bbT) \text{ factors through } X \Rightarrow \mathfrak{A}_{\widetilde \nu}([\widetilde\phi]) \in \bbZ.
\end{equation}

Define $\widetilde B^E_t$ for arbitrary $t \in \bbR$ by using a homotopy to the identity in a way similar to \eqref{e.interpolation}. Then use $\widetilde B^E_t$ to produce a continuous section $\widetilde \Lambda^+:\widetilde X \to \bbR\bbP^1$ by
\begin{equation} \label{eq:Lambda+interpolatedDef}
\widetilde\Lambda^+([\widetilde\omega,s]) = \widetilde B^E_s \widetilde\Lambda^+(\widetilde\omega), \quad \widetilde\omega \in \widetilde\Omega, \ s \in [0,1].
\end{equation}

The reader may then verify (see \cite{DFdm} for details) that the map $\widetilde\Lambda^+$ from  \eqref{eq:Lambda+interpolatedDef} is well-defined and continuous, and that it depends only on the first coordinate of $\widetilde \omega$, that is, there exists a continuous map $\Lambda^+:\bbT \to \bbR\bbP^1$ such that
$\widetilde\Lambda^+(\widetilde \omega) = \Lambda^+(\omega)$ for every $\widetilde \omega = (\omega,x,y) \in \widetilde\Omega$.

Thus, $\widetilde\Lambda^+$ factors through $X$, and hence $\mathfrak{A}_{\widetilde \nu}(\widetilde\Lambda^+) \in \bbZ$ by \eqref{e.solenoidThroughTCsharp}. Arguing as in the proof of Theorem~\ref{t.main}, $1-\widetilde k _{\widetilde p, \widetilde q}(E) = \mathfrak{A}_{\widetilde\nu}(\widetilde\Lambda^+)$, so $\mathfrak{A}_{\widetilde \nu}(\widetilde\Lambda^+) \in \bbZ$ yields \eqref{e.goal:idsmustbeinteger} and completes the proof.
\end{proof}

\begin{remark}\label{r.concludingrem}
Let us the address the obvious question whether Theorem~\ref{t.doublingmap} still holds if the assumption $p \in C(\bbT, \bbC \setminus \{ 0 \})$ is dropped. In short, this is unclear to us. Two statements are fundamentally true: 

(i) Our proof relies on the fact that the stable section depends only on the future, that is, on the forward iterates of a point $\tilde \omega = (\omega,x,y)$ under $\tilde T$, and in addition the sampling functions evaluated along the forward iterates then actually only depend on $\omega$ (i.e., they are independent of $x,y$). This in turn allows us to view the stable section as being defined on $\bbT$, which in turn leads to a ``trivial'' topology. 

(ii) As soon as $p$ is allowed to take the value zero, the stable section will no longer be defined as a map into $\bbR \bbP^1$, then interpreted as a map into $\bbT$, and hence it cannot be used to view the associated rotation number at the energy in question from the perspective of the Schwartzman homomorphism. 

Accepting (i) and (ii) one is naturally limited to working with the unstable section, but the latter clearly depends on the past, and hence on the components $x$ and $y$ of $\tilde \omega$ as well. Consequently, one cannot reduce to maps with trivial topology in order to evaluate the rotation number in a gap. In other words, it is unclear whether it necessarily takes integer values there, and hence the extension of Theorem~\ref{t.doublingmap} to general $p \in C(\bbT,\bbC)$ remains an interesting open problem.
\end{remark}

\bibliographystyle{abbrv}

\bibliography{gapbib}

\end{document}